\documentclass[11pt]{article}
\usepackage[T1]{fontenc}
\usepackage{lmodern}
\usepackage[margin=1.05in]{geometry}
\usepackage{amsmath,amssymb,amsthm,mathtools}
\usepackage{mathrsfs}
\usepackage{microtype}
\usepackage[hidelinks]{hyperref}
\usepackage{enumitem}
\setlist{topsep=4pt,itemsep=2pt}
\numberwithin{equation}{section}
\newtheorem{theorem}{Theorem}[section]
\newtheorem{proposition}[theorem]{Proposition}
\newtheorem{lemma}[theorem]{Lemma}

\theoremstyle{remark}

\DeclareMathOperator{\supp}{supp}
\DeclareMathOperator{\dist}{dist}
\DeclareMathOperator{\diam}{diam}

\newcommand{\dd}{\,d}
\newcommand{\R}{\mathbb R}
\newcommand{\F}{\mathbb F_2}
\newcommand{\Neh}{\mathcal N}
\newcommand{\Prob}{\mathcal P}
\newcommand{\Um}{\mathscr U}
\newcommand{\ip}[2]{\langle #1,#2\rangle}
\title{Positive solutions of Yamabe-type equations with general potentials}
\author{%
Hongyu Chen\thanks{School of Mathematics, Sichuan University,
Chengdu 610064, China.
Email: \texttt{hongyu.chern@gmail.com}.}
\qquad
Li Ma\thanks{Corresponding author.
School of Mathematics and Physics,
University of Science and Technology Beijing,
Beijing 100083, China.
Email: \texttt{nuslma@126.com}.}
}
\date{September, 2026}
\begin{document}
\maketitle

\begin{abstract}
Let $(M^n,g)$ be a smooth closed connected Riemannian manifold of dimension $n\ge3$, and let $\psi\in C^\infty(M,\R)$. We prove that the equation
$$
-\frac{4(n-1)}{n-2}\Delta_g u+\psi u=u^{\frac{n+2}{n-2}}
$$
has a smooth positive solution if and only if its linear part has positive first eigenvalue. In dimensions $n\ge6$, the proof combines an energy estimate under exact local mass constraints with an intersection argument for finite cycles of atomic measures; dimensions $3\le n\le5$ follow from the theorem of Bahri and Brezis.



\end{abstract}

\section{Introduction}\label{sec:introduction}

The famous question from the work of Barhi-Brezis \cite{BahriBrezis1996} is below. Given a smooth function $\psi$ on a smooth closed connected Riemannian manifold $(M^n,g)$ of dimension $n\ge3$. Define the operator
$L_\psi=-\frac{4(n-1)}{n-2}\Delta_g+\psi$, where we use the convention $\Delta_g=\operatorname{div}_g\nabla_g$, 
and its first eigenvalue
$$
\lambda_1(L_\psi)=\inf \{ \int_M uL_\psi (u)\dd V_g; \ u\in C^\infty(M,\R), \ \ \int_Mu^2\dd V_g=1\}.
$$
Bahri and Brezis asked that if there is a positive smooth solution to the elliptic problem
\begin{equation}\label{eq:main}
-\frac{4(n-1)}{n-2}\Delta_g u+\psi u=u^{\frac{n+2}{n-2}},\qquad u>0,
\end{equation}
under the spectral condition $\lambda_1(L_\psi)>0$. The difficulty for the study of the problem \eqref{eq:main} is the loss of compactness because the Sobolev embedding is not compact. However, this kind of problem has Struwe's compactness and related stability theory, which are presented in detail in the famous book \cite{DHR04}. The heat flow method is considered in \cite{Ma2023}.

For dimensions $3\le n\le5$, the affirmative answer follows from the theorem of Bahri and Brezis \cite{BahriBrezis1996}. We show in this work that for dimension bigger than $5$, the answer is still affirmative. Note that
The spectral condition $\lambda_1(L_\psi)>0$ is necessary: testing \eqref{eq:main} against a positive first eigenfunction $\varphi_1$ gives
$$
\lambda_1(L_\psi)\int_Mu\varphi_1\dd V_g
=\int_Mu^{\frac{n+2}{n-2}}\varphi_1\dd V_g>0.
$$
So the issue we address is the sufficiency of this spectral condition.

For $\psi=R_g$, where $R_g$ is the scalar curvature, \eqref{eq:main} is the positive Yamabe equation. Under the spectral condition, Aubin \cite{Aubin1976} and Schoen \cite{Schoen1984} established its solvability through comparison with the spherical Sobolev constant and the positive mass theorem; see Lee and Parker \cite{LeeParker1987}. An arbitrary potential preserves the critical nonlinearity but removes the scalar-curvature transformation law used in these arguments. For blow-up, localization and compactness theory for Yamabe-type equations, see the survey of Druet and Hebey \cite{DruetHebey2005}.

Under the spectral condition, minimization of the Yamabe type functional
$$
Y_\psi(v)=\frac{\displaystyle\int_M\bigl(\frac{4(n-1)}{n-2}|\nabla v|^2+\psi v^2\bigr)\dd V_g}
{\displaystyle\left(\int_M|v|^{2n/(n-2)}\dd V_g\right)^{(n-2)/n}},
\qquad v\in H^1(M)\setminus\{0\},
$$
produces a positive solution when the infimum is strictly below the spherical value. At the critical exponent, minimizing sequences can instead concentrate at Euclidean bubbles. In dimensions at least six, the leading correction to a bubble's energy is a positive multiple of $(\psi-R_g)(P)\delta^2$, where $P$ and $\delta$ are its center and scale. A negative value of $\psi-R_g$ therefore gives a test function below the first concentration level. When $\psi-R_g\ge0$, this second-order expansion does not provide a strict inequality. On the round sphere with $\psi=R_g+c$ and $c>0$, every nonzero function has quotient strictly above the spherical value, while concentrating bubbles approach it. The infimum is not attained, although the equation has a positive constant solution. The existence problem therefore requires a positive critical point which need not be a minimizer.

Critical-function methods address attainment of sharp Sobolev bounds under criticality and geometric hypotheses \cite{HebeyVaugon2001,Collion2007}; existence results based on improved invariant Sobolev embeddings impose symmetry and coefficient assumptions \cite{Djadli1999Symmetries,ClappFernandez2017}. Small linear perturbations of $R_g$ on manifolds of positive Yamabe type also admit blowing-up families under dimensional and geometric conditions \cite{EspositoPistoiaVetois2014}.

Applications of the barycenter method to critical boundary problems include the work of Mayer and Ndiaye \cite{MayerNdiaye2017} on the Cherrier--Escobar problem, which combines bubble interaction estimates with relative homology. For non-contractible bounded smooth domains in $\R^3$, Aldawood and Ndiaye \cite{AldawoodNdiaye2025} obtained positive solutions of the critical Dirichlet equation with a smooth linear potential under positivity and maximum-principle assumptions on the Dirichlet operator.

Ma \cite[Theorems 1 and 2 in the preprint version]{Ma2023} constructed an $L^{2n/(n-2)}$-preserving parabolic flow for the same potential. He proved global existence for smooth positive initial data and obtained a smooth positive stationary limit along a sequence of times when the Sobolev infimum is positive and strictly below the spherical value, with initial quotient at most that value.

Bahri and Brezis \cite[Theorem 1]{BahriBrezis1988} announced that every coercive operator with a bounded real potential on a closed manifold of dimension three, four or five admits a positive solution of the critical equation. Their complete proof appears in \cite[Theorem 2, p.~41]{BahriBrezis1996}. In arbitrary dimension they also proved existence when $H_1(M;\mathbb Z_2)$ or $H_2(M;\mathbb Z_2)$ is nonzero, and for $n=6$ or $7$ when $H_3(M;\mathbb Z_2)$ is nonzero; see \cite[Theorem 2]{BahriBrezis1988} and \cite[Theorem 3, p.~59]{BahriBrezis1996}. In \cite[Remarque 2, p.~573]{BahriBrezis1988}, they conjectured that the spectral condition alone suffices in every dimension $n\ge6$. The substitution $u=\left(\frac{4(n-1)}{n-2}\right)^{(n-2)/4}v$ changes \eqref{eq:main} into $-\Delta_gv+\frac{(n-2)\psi}{4(n-1)}v=v^{\frac{n+2}{n-2}}$, which is their normalization.

\begin{theorem}\label{thm:main}
Let $(M^n,g)$ be a smooth closed connected Riemannian manifold, $n\ge3$, and let $\psi\in C^\infty(M,\R)$. Then $\lambda_1(L_\psi)>0$ if and only if \eqref{eq:main} has a solution $u\in C^\infty(M)$ with $u>0$.
\end{theorem}

The high-dimensional assertion of Theorem~\ref{thm:main} is the smooth-potential version of the conjecture in \cite[Remarque 2]{BahriBrezis1988}, whose original formulation allows bounded potentials.

The proof combines the topological variational approach of Bahri and Coron \cite{BahriCoron1988} and Bahri and Brezis \cite{BahriBrezis1996} with Mazumdar's manifold bubble decomposition \cite{Mazumdar2016}, in the concentration and decomposition framework of Lions \cite{Lions1985} and Struwe \cite{Struwe1984}. Quantitative estimates for Euclidean multi-bubble decompositions were obtained by Figalli and Glaudo \cite{FigalliGlaudo2020} and Deng, Sun and Wei \cite{DengSunWei2025}. For two or more bubbles, the sharp control of the distance by the equation residual is linear in dimensions three to five, involves a logarithmic correction in dimension six, and a fractional power in higher dimensions. The essential case is $\psi-R_g\ge0$ with strict positivity somewhere. Near a fixed separated $K$-point configuration in this positive set, exact equality of the local critical masses forces the Nehari energy above the $K$-bubble level for arbitrary functions satisfying the constraints. A finite cycle of atomic measures supported in common fibers of a piecewise linear map to $\R^{n-1}$ has mod-two intersection one with a compact family of such configurations. A local center and weight selector transfers this pairing to a relative class of Nehari sublevels. Each fiber lies in a curve of uniformly bounded length, which allows the corresponding positive test cycle to be deformed below that level while remaining below the next bubble level. Under the no-solution assumption, integer energy quantization and a positive descending flow make these conclusions incompatible.

In Section~\ref{sec:variational} we establish regularity, the positive descending flow, energy quantization and finite-time localization. In Section~\ref{sec:local} we prove the estimate under exact local mass constraints. Section~\ref{sec:topology} constructs the finite cycles and computes their intersection pairing. Section~\ref{sec:tests} realizes these cycles by positive test functions and estimates their energy homotopy. In Section~\ref{sec:detector} we construct the relative energy class and prove Theorem~\ref{thm:main}.

\section{The variational space and a positive descending flow}\label{sec:variational}

For the remainder of the paper we assume $\lambda_1(L_\psi)>0$; dimension restrictions are stated in the relevant sections. In this section $n\ge3$. Define
\begin{align*}
Q(v,w)&=\int_M\bigl(\frac{4(n-1)}{n-2}\ip{\nabla v}{\nabla w}+\psi vw\bigr)\dd V_g,\\
J(v)&=\frac12Q(v,v)-\frac{n-2}{2n}\int_M(v_+)^{\frac{2n}{n-2}}\dd V_g.
\end{align*}
All function spaces and integrals use $g$. Inequalities and nonnegativity of Sobolev functions are understood almost everywhere. A full Palais--Smale sequence at level $c$ means a sequence with $J(u_j)\to c$ and $J'(u_j)\to0$ in $H^{-1}(M)$, with derivatives tested in the entire space $H^1(M)$. The inequalities
$$
Q(v,v)\ge\lambda_1(L_\psi)\|v\|_2^2,
\qquad
\frac{4(n-1)}{n-2}\|\nabla v\|_2^2\le Q(v,v)+\|\psi_-\|_\infty\|v\|_2^2
$$
show that $Q$ is an inner product equivalent to the $H^1$ inner product. We write $\|v\|_Q^2=Q(v,v)$.

For $v\in H^1(M)$ nonzero and nonnegative, let
\begin{equation}\label{eq:projection}
\Pi(v)=\left(\frac{Q(v,v)}{\int_Mv^{\frac{2n}{n-2}}\dd V_g}\right)^{(n-2)/4}v.
\end{equation}
This is the unique point on the ray through $v$ in
$$
\Neh=\left\{u\in H^1(M):u\ge0,\ u\ne0,\ Q(u,u)=\int_Mu^{\frac{2n}{n-2}}\dd V_g\right\}.
$$
The map $\Pi$ and normalization by the $Q$ norm give inverse homeomorphisms between $\Neh$ and $\Sigma=\{z\in H^1(M):z\ge0,\ \|z\|_Q=1\}$. Direct maximization along a ray gives
\begin{equation}\label{eq:ray-energy}
J(\Pi(v))=\max_{t>0}J(tv)
=\frac{Q(v,v)^{n/2}}{n\left(\int_Mv^{\frac{2n}{n-2}}\dd V_g\right)^{(n-2)/2}}.
\end{equation}
We use strict sublevels $\Neh^b=\{u\in\Neh:J(u)<b\}$. Sobolev embedding gives a positive lower bound for $J$ on $\Neh$.

The Euclidean solution and its energy are
\begin{align*}
U(x)&=[4n(n-1)]^{(n-2)/4}(1+|x|^2)^{-(n-2)/2},\\
e_n&=\frac1n\int_{\R^n}U^{\frac{2n}{n-2}}\dd x
=\frac{[4n(n-1)]^{n/2}\pi^{n/2}\Gamma(n/2)}{n\Gamma(n)}.
\end{align*}
Radial differentiation gives $-\frac{4(n-1)}{n-2}\Delta U=U^{\frac{n+2}{n-2}}$. The integral follows by setting $s=r^2$ in the radial beta integral. In particular, $e_6=4800\pi^3$.

\begin{lemma}\label{lem:regularity}
Every nonzero nonnegative weak $H^1$ solution of $L_\psi u=u^{\frac{n+2}{n-2}}$ is smooth and strictly positive.
\end{lemma}

\begin{proof}
We use the truncation argument of Brezis and Kato \cite[proof of Theorem~2.3]{BrezisKato1979}. The equation is $-\frac{4(n-1)}{n-2}\Delta_gu=(u^{\frac{4}{n-2}}-\psi)u$, and $u^{\frac{4}{n-2}}-\psi\in L^{n/2}(M)$. For $\beta\ge1$ and $T>0$, test the weak equation with $u\min(u,T)^{2\beta-2}$ and set $w_T=u\min(u,T)^{\beta-1}$. Truncation and approximation justify this test. On $\{u<T\}$ the ratio of the testing gradient term to $|\nabla w_T|^2$ is $(2\beta-1)/\beta^2$, and on $\{u\ge T\}$ it is one. Hence
$$
\frac{4(n-1)(2\beta-1)}{(n-2)\beta^2}\int_M|\nabla w_T|^2\dd V_g
\le\int_M(u^{\frac{4}{n-2}}-\psi)_+w_T^2\dd V_g.
$$
Split the coefficient on the right into a bounded function and a tail with arbitrarily small $L^{n/2}$ norm. Hölder's inequality and the Sobolev inequality absorb the tail, for this fixed $\beta$, and give $\|w_T\|_{\frac{2n}{n-2}}^2\le C_\beta\|w_T\|_2^2$. If $u\in L^{2\beta}$, monotone convergence therefore gives $u\in L^{\frac{2n\beta}{n-2}}$. Starting with $2\beta=\frac{2n}{n-2}$ multiplies the available integrability exponent by $\frac{n}{n-2}>1$ at each iteration. Thus $u$ belongs to every finite $L^r$ space, and so does $u^{\frac{n+2}{n-2}}-\psi u$. Elliptic estimates give $u\in W^{2,r}(M)$ for every finite $r$. Taking $r>n$ yields $u\in C^{1,\alpha}$ for some $\alpha>0$. The map $s\mapsto s^{\frac{n+2}{n-2}}$ is locally Lipschitz on $[0,\infty)$, so the right-hand side is H\"older continuous; Schauder estimates give $u\in C^{2,\alpha}$.

Choose $C\ge0$ with $\psi+C\ge0$. The inequality $(-\frac{4(n-1)}{n-2}\Delta_g+\psi+C)u=u^{\frac{n+2}{n-2}}+Cu\ge0$ and the strong maximum principle imply $u>0$. Its minimum on $M$ is positive. The power $u^{\frac{n+2}{n-2}}$ is consequently a smooth composition, so elliptic bootstrapping gives $u\in C^\infty(M)$.
\end{proof}

\begin{lemma}\label{lem:flow}
There is a continuous flow $\Phi_t:\Neh\to\Neh$, defined for every $t\ge0$, which decreases $J$. On $\Sigma$ it is induced by
$$
z'=V(z),\qquad
V(z)=L_\psi^{-1}(z^{\frac{n+2}{n-2}})-\left(\int_Mz^{\frac{2n}{n-2}}\dd V_g\right)z.
$$
On every fixed finite positive energy interval, a sequence with $\|V(z_j)\|_Q\to0$ gives, after passing to a subsequence, a bounded full Palais--Smale sequence $\Pi(z_j)$ for $J$ on $H^1(M)$.
\end{lemma}

\begin{proof}
Lax--Milgram defines $L_\psi^{-1}:H^{-1}\to H^1$. It preserves nonnegativity: if $L_\psi w=f\ge0$, testing with $w_-$ gives $-Q(w_-,w_-)=\ip{f}{w_-}\ge0$. Hölder's inequality and the mean-value formula give
$$
\|z_+^{\frac{n+2}{n-2}}-y_+^{\frac{n+2}{n-2}}\|_{\frac{2n}{n+2}}
\le(\frac{n+2}{n-2})(\|z\|_{\frac{2n}{n-2}}+\|y\|_{\frac{2n}{n-2}})^{\frac{4}{n-2}}\|z-y\|_{\frac{2n}{n-2}}.
$$
To construct the flow before establishing invariance of the nonnegative cone, use on a $Q$ annulus the field
$$
z\longmapsto L_\psi^{-1}((z_+)^{\frac{n+2}{n-2}})
-\frac{\int_M(z_+)^{\frac{2n}{n-2}}\dd V_g}{Q(z,z)}z.
$$
The preceding power estimate makes it locally Lipschitz. Its $Q$ inner product with $z$ vanishes, so each $Q$ sphere is invariant. Sobolev embedding bounds the field uniformly on the unit sphere. Every finite-time trajectory is therefore Lipschitz in $H^1$, has a limit at a finite endpoint, and extends by local existence. Variation of constants gives, for a trajectory starting in $\Sigma$,
\begin{align*}
z(t)={}&\exp\left(-\int_0^t\int_M(z(s)_+)^{\frac{2n}{n-2}}\dd V_g\dd s\right)z(0)\\
&+\int_0^t\exp\left(-\int_s^t\int_M(z(r)_+)^{\frac{2n}{n-2}}\dd V_g\dd r\right)
L_\psi^{-1}((z(s)_+)^{\frac{n+2}{n-2}})\dd s.
\end{align*}
All forcing terms are nonnegative by positivity of the inverse. Hence $z(t)\ge0$, the positive parts can be removed, and the field reduces to $V$ on $\Sigma$. Conjugating this flow by $\Pi$ defines $\Phi_t$ on $\Neh$.

Since $Q(L_\psi^{-1}(z^{\frac{n+2}{n-2}}),z)=\int_Mz^{\frac{2n}{n-2}}\dd V_g$ and $Q(z,V)=0$, differentiation gives
\begin{equation}\label{eq:flow-energy}
\begin{split}
\frac{d}{dt}\int_Mz^{\frac{2n}{n-2}}\dd V_g&=\frac{2n}{n-2}\|V(z)\|_Q^2,\\
\frac{d}{dt}J(\Pi(z))&=-\left(\int_Mz^{\frac{2n}{n-2}}\dd V_g\right)^{-n/2}\|V(z)\|_Q^2.
\end{split}
\end{equation}
The full ambient gradient satisfies
\begin{equation}\label{eq:full-gradient}
\nabla_QJ(\Pi(z))
=-\left(\int_Mz^{\frac{2n}{n-2}}\dd V_g\right)^{-(n+2)/4}V(z).
\end{equation}
On a finite positive energy interval, \eqref{eq:ray-energy} bounds $\int_Mz^{\frac{2n}{n-2}}\dd V_g$ above and below by positive constants. Formula~\eqref{eq:full-gradient} proves the last assertion after taking an energy-convergent subsequence.
\end{proof}

Let $\Prob(M)$ denote the probability measures on $M$, equipped with the Wasserstein distance $W_1$ for the metric $g$. Since $M$ is compact, this distance induces weak convergence and makes $\Prob(M)$ compact. For $u\ge0$, $u\ne0$, define
\begin{equation}\label{eq:mass-measure}
\mu_u=\frac{u^{\frac{2n}{n-2}}\dd V_g}{\int_Mu^{\frac{2n}{n-2}}\dd V_g}.
\end{equation}
For an integer $K\ge1$, let $B_K(M)\subset\Prob(M)$ consist of measures with at most $K$ support points. These spaces of formal barycenters and their topology are studied by Kallel and Karoui \cite{KallelKaroui2011}.

\begin{lemma}\label{lem:quantization}
Assume that \eqref{eq:main} has no smooth positive solution. If $(u_j)$ is a nonnegative bounded full Palais--Smale sequence for $J$ at a positive level, then, after a subsequence, there is an integer $k\ge1$ such that
$$
J(u_j)\longrightarrow ke_n,\qquad
\mu_{u_j}\longrightarrow\frac1k\sum_{i=1}^k\delta_{P_i}
$$
in $W_1$, for points $P_i\in M$ which need not be distinct.
\end{lemma}

\begin{proof}
For an orientable manifold apply Mazumdar's decomposition \cite[Theorem 1.2 in the cited preprint version]{Mazumdar2016} with order one, no boundary and zeroth-order coefficient $\frac{(n-2)\psi}{4(n-1)}$. Its functional is
$$
I(v)=\frac12\int_M\left(|\nabla v|^2+\frac{(n-2)\psi}{4(n-1)}v^2\right)\dd V_g
-\frac{n-2}{2n}\int_M|v|^{\frac{2n}{n-2}}\dd V_g.
$$
The substitution $u=\left(\frac{4(n-1)}{n-2}\right)^{(n-2)/4}v$ gives $J(u)=\left(\frac{4(n-1)}{n-2}\right)^{n/2}I(v)$ on nonnegative functions and converts full Palais--Smale sequences between the two functionals. At a nonnegative function, the first derivatives of the positive-part functional $J$ and the corresponding absolute-value functional agree in every ambient direction. The cited theorem permits an arbitrary smooth real zeroth-order coefficient, and its condition on the order is $2<n$. Since the manifold has no boundary, only entire Euclidean profiles occur. For a nonnegative Palais--Smale sequence it gives a nonnegative weak solution, finitely many positive Euclidean bubbles, a remainder tending to zero in $H^1$, and additive energy splitting. By Lemma~\ref{lem:regularity} the weak solution must be zero under the present assumption. The classification of positive entire solutions by Caffarelli, Gidas and Spruck \cite{CaffarelliGidasSpruck1989} identifies every Euclidean profile with a translate and dilate of $U$ after the above normalization. Each bubble therefore has energy $e_n$ and critical mass $ne_n$. Thus
$$
u_j=\sum_{i=1}^k U_{i,j}+o_{H^1}(1),\qquad J(u_j)\to ke_n.
$$
Here $U_{i,j}$ are nonnegative cut-off Euclidean bubbles. Testing $J'(u_j)$ against $u_j$ gives $\int_Mu_j^{\frac{2n}{n-2}}\dd V_g\to kne_n$. The $H^1$ remainder is small in $L^{\frac{2n}{n-2}}$, so the same limit holds for $\int_M(\sum_iU_{i,j})^{\frac{2n}{n-2}}\dd V_g$. Since
$$
\left(\sum_iU_{i,j}\right)^{\frac{2n}{n-2}}-\sum_iU_{i,j}^{\frac{2n}{n-2}}\ge0
$$
and its integral tends to zero, this difference tends to zero in $L^1$. Testing against continuous functions proves the measure assertion, including coincident centers and unequal bubble scales.

For completeness, suppose that $M$ is nonorientable and let $\pi:\widetilde M\to M$ be its orientation double cover. If $\varphi_1>0$ is a downstairs first eigenfunction, integration by parts gives, for every $v\in H^1(\widetilde M)$,
$$
\widetilde Q(v,v)=\frac{4(n-1)}{n-2}\int_{\widetilde M}(\varphi_1\circ\pi)^2
\left|\nabla\frac{v}{\varphi_1\circ\pi}\right|^2\dd V_{\widetilde g}
+\lambda_1(L_\psi)\int_{\widetilde M}v^2\dd V_{\widetilde g}.
$$
This proves coercivity on the entire upstairs space. The lifted sequence is full Palais--Smale upstairs: the derivative vanishes on the odd part of any test function by deck invariance, and its even part descends. Its weak limit is invariant and would descend to a positive solution if nonzero. The upstairs quantum measure is deck invariant. Since the deck involution is free, its atoms occur in pairs with equal integer multiplicities. Dividing energy and mass by two yields precisely the asserted integer quantization downstairs.
\end{proof}

\begin{lemma}\label{lem:gap}
Under the no-solution assumption of Lemma~\ref{lem:quantization}, $\|V\|_Q$ has a positive lower bound on every closed finite positive energy interval disjoint from $\{ke_n:k\ge1\}$. If $0<\alpha<\beta$ and $[\alpha,\beta]$ contains no such level, there is a single $T<\infty$ with $\Phi_T(\Neh^\beta)\subset\Neh^\alpha$. Every lower sublevel is preserved throughout this flow homotopy.
\end{lemma}

\begin{proof}
Otherwise \eqref{eq:full-gradient} gives a full Palais--Smale sequence whose limiting energy lies in that interval, contradicting Lemma~\ref{lem:quantization}. Sobolev embedding bounds $\int_Mz^{\frac{2n}{n-2}}\dd V_g$ above on $\Sigma$. Thus \eqref{eq:flow-energy} has a fixed negative upper bound while the trajectory is in the interval. Integrating this bound gives a common finite time; below the target level monotonicity keeps the trajectory there. The same flow gives the asserted homotopies.
\end{proof}

\begin{proposition}\label{prop:localization}
Assume that \eqref{eq:main} has no positive solution. Fix $K\ge1$, $Ke_n<b_*<(K+1)e_n$ and $0<\rho_1<\rho_2$. There is $\varepsilon_0>0$ such that, for $0<\varepsilon<\varepsilon_0$, a common finite flow time maps the pair
$$
(X,A)=(\Neh^{Ke_n+\varepsilon},\Neh^{Ke_n-\varepsilon})
$$
to $(Z,A)$, where
$$
Z=A\cup\{u\in X:\dist_{W_1}(\mu_u,B_K(M))<\rho_2\}.
$$
For a source initially satisfying $\dist_{W_1}(\mu_u,B_K(M))<\rho_1$, the entire flow homotopy lies in $Z$.
\end{proposition}

\begin{proof}
On $\Sigma\cap\{J\circ\Pi\le b_*\}$ the denominator in \eqref{eq:mass-measure} is bounded below. The power estimate
$$
\int_M|z^{\frac{2n}{n-2}}-y^{\frac{2n}{n-2}}|\dd V_g
\le \frac{2n}{n-2}(\|z\|_{\frac{2n}{n-2}}+\|y\|_{\frac{2n}{n-2}})^{\frac{n+2}{n-2}}\|z-y\|_{\frac{2n}{n-2}}
$$
and normalization show that $z\mapsto\mu_z$ is uniformly Lipschitz from the $Q$ norm to $W_1$ on this set. The distance to $B_K(M)$ is therefore also Lipschitz, with a fixed constant $C$.

For each $\rho>0$, $\|V\|_Q\ge\eta_\rho>0$ on the part of the energy cap outside the tube of radius $\rho$. Failure would give, by Lemmas~\ref{lem:flow} and \ref{lem:quantization}, an atomic limit with at most $K$ points. The positive lower energy bound excludes level zero.

On the energy cap, \eqref{eq:flow-energy} gives $-(J\circ\Pi)'\ge c_0\|V\|_Q^2$. During a crossing from the distance $\rho_1$ to the distance $\rho_2$ that stays outside the inner tube, the energy loss is at least
$$
c_0\eta_{\rho_1}\int\|z'\|_Q\dd t
\ge\frac{c_0\eta_{\rho_1}}C(\rho_2-\rho_1)>0.
$$
Choose $2\varepsilon_0$ smaller than this last quantity, with $Ke_n+\varepsilon_0<b_*$ and $Ke_n-\varepsilon_0>0$. A trajectory starting in the inner tube cannot leave the outer tube before reaching $A$. If an arbitrary trajectory starting in $X$ is outside both $A$ and the outer tube at time $T$, it never previously entered the closed inner tube; otherwise the preceding crossing estimate applies. Its energy decay has therefore been at least $c_0\eta_{\rho_1}^2$ throughout. Choosing $T>2\varepsilon/(c_0\eta_{\rho_1}^2)$ is a contradiction. This proves both assertions. All maps are a single continuous finite-time flow on the indicated spaces.
\end{proof}

\section{Concentration with exactly equal local masses}\label{sec:local}

In this section $n\ge6$. Choose a smooth radial cutoff $\chi$ which is one near zero and has support in a fixed ball smaller than the injectivity radius. At its vanishing endpoint choose it flat, so every positive power of $\chi$ is smooth. In geodesic normal coordinates put
\begin{equation}\label{eq:cutoff-bubble}
U_{\delta,P}(x)=\chi(d_g(P,x))\delta^{-(n-2)/2}
U\bigl(\delta^{-1}\exp_P^{-1}(x)\bigr).
\end{equation}
The radial expression makes this independent of the orthonormal frame at $P$; it is extended by zero outside the coordinate ball. Constants in local estimates may depend on the cutoff, $n$, $M$, $g$ and $\psi$.

\begin{lemma}\label{lem:local-expansion}
Uniformly for $P\in M$ as $\delta\to0$,
\begin{equation}\label{eq:local-energy}
J(U_{\delta,P})
=e_n+\frac12(\psi-R_g)(P)\delta^2\int_{\R^n}U^2\dd x+o(\delta^2).
\end{equation}
The same expansion holds for $J(\Pi(U_{\delta,P}))$. Moreover,
\begin{equation}\label{eq:residual}
\|L_\psi U_{\delta,P}-U_{\delta,P}^{\frac{n+2}{n-2}}\|_{H^{-1}}
\le
\begin{cases}
C\delta^2(1+|\log\delta|)^{2/3},&n=6,\\
C\delta^2,&n\ge7.
\end{cases}
\end{equation}
\end{lemma}

\begin{proof}
The beta integral and integration by parts against $|x|^2U$ give
\begin{align*}
\int_{\R^n}U^2\dd x
&=[4n(n-1)]^{(n-2)/2}\pi^{n/2}
\frac{\Gamma(n/2-2)}{\Gamma(n-2)},\\
\frac{4(n-1)}{n-2}\int_{\R^n}|x|^2|\nabla U|^2\dd x
&=n(n+2)\int_{\R^n}U^2\dd x,\\
\int_{\R^n}|x|^2U^{\frac{2n}{n-2}}\dd x
&=\frac{n^2(n-4)}{n-2}\int_{\R^n}U^2\dd x.
\end{align*}
All these moments are finite for $n\ge6$. The expansion must retain the scalar-curvature term, since its sign is decisive below. By the Gauss lemma, the squared metric gradient of a radial function is exactly the square of its radial derivative. The angular average of the normal-coordinate volume density is
$$
1-\frac{R_g(P)}{6n}r^2+O(r^4).
$$
Odd Taylor terms have zero angular average. The linear potential term also integrates to zero. Substitution in $J$ therefore gives the second-order curvature coefficient
$$
-\frac{n+2}{12}\int_{\R^n}U^2\dd x
+\frac{n-4}{12}\int_{\R^n}U^2\dd x
=-\frac12\int_{\R^n}U^2\dd x.
$$
The potential contributes $\psi(P)\delta^2\int_{\R^n}U^2\dd x/2$. The gradient cutoff error is $O(\delta^{n-2})$, and the critical-mass cutoff error is $O(\delta^n)$. For $n\ge7$ the fourth gradient moment and the second moment of $U^2$ are finite, so the other remainders are $O(\delta^4)$. For $n=6$, the same truncated moments are $O(1+|\log\delta|)$ and yield $o(\delta^2)$. This proves \eqref{eq:local-energy}. The individual integrals $Q(U_{\delta,P},U_{\delta,P})$ and $\int_MU_{\delta,P}^{\frac{2n}{n-2}}\dd V_g$ equal $ne_n+O(\delta^2)$. Thus the ray parameter in \eqref{eq:projection} is $1+O(\delta^2)$, and changing to that parameter changes the energy by $O(\delta^4)$.

In the cutoff core, put $r=d_g(P,x)$. Then $\Delta_gr=(n-1)/r+O(r)$ and $r|\partial_r U_{\delta,P}|\le(n-2)U_{\delta,P}$. Hence the residual is bounded by $CU_{\delta,P}$ there. On the fixed cutoff annulus it is $O(\delta^{(n-2)/2})$. The inclusion $L^{2n/(n+2)}\subset H^{-1}$ and rescaling now give \eqref{eq:residual}. The rescaled core integral is finite for $n>6$. For $n=6$ it grows logarithmically, with norm exponent $2/3$. In particular, for any fixed finite collection of scales tending to zero,
\begin{equation}\label{eq:residual-small-square}
\left(\sum_i\|L_\psi U_{\delta_i,P_i}-U_{\delta_i,P_i}^{\frac{n+2}{n-2}}\|_{H^{-1}}\right)^2
=o\left(\sum_i\delta_i^2\right).
\end{equation}
\end{proof}

\begin{lemma}\label{lem:normal}
Fix $K$ and a compact family of center configurations with pairwise distances bounded below. For sufficiently small scales, if $v$ is $Q$-orthogonal to $U_{\delta_i,P_i}$ and to its scale and center derivatives for every $i$, then
\begin{equation}\label{eq:normal-coercivity}
Q(v,v)-\frac{n+2}{n-2}\sum_{i=1}^K\int_MU_{\delta_i,P_i}^{\frac{4}{n-2}}v^2\dd V_g
\ge c\|v\|_Q^2
\end{equation}
for a constant $c>0$. Here the cutoffs are chosen sufficiently small to have disjoint supports. The constants may depend on the fixed $K$ and center family; they are independent of the ratios of the bubble scales.
\end{lemma}

\begin{proof}
For a single Euclidean bubble, stereographic conformal covariance converts the gradient inner product to the quadratic form of $-\frac{4(n-1)}{n-2}\Delta_{S^n}+n(n-1)$. The bubble amplitude corresponds to constant functions, and its translation and scale derivatives correspond to first spherical harmonics. The first remaining Laplace eigenvalue is $2(n+1)$. Consequently, on their gradient-orthogonal complement,
$$
\frac{n+2}{n-2}\int_{\R^n}U^{\frac{4}{n-2}}f^2\dd x
\le\frac{n(n+2)}{8(n+1)+n(n-2)}
\frac{4(n-1)}{n-2}\int_{\R^n}|\nabla f|^2\dd x
=\frac{4n(n-1)}{(n+4)(n-2)}\int_{\R^n}|\nabla f|^2\dd x.
$$

Suppose the stated estimate fails along scales tending to zero, and normalize $\|v_j\|_Q=1$. Pass to a weak unscaled limit $v_\infty$ and set $w_j=v_j-v_\infty$. Rellich compactness and the Hilbert norm splitting give
$$
w_j\to0\ \hbox{in }L^2(M),\qquad
Q(w_j,w_j)=1-Q(v_\infty,v_\infty)+o(1).
$$
Coercivity gives $Q(v_\infty,v_\infty)\ge0$. Rescale $w_j$ about each center at its own bubble scale. On every fixed rescaled ball a diagonal subsequence converges weakly in $\dot H^1$ and strongly in local $L^2$ to a function $f_i$. The normalized orthogonality conditions converge to the Euclidean ones: the normalized parameter derivatives converge strongly in the rescaled gradient norm, and the energy of the fixed function $v_\infty$ on shrinking balls tends to zero. The fixed weak limit has vanishing weighted contribution: on a small center ball use absolute continuity of its $L^{\frac{2n}{n-2}}$ integral, and on the complement use the vanishing $L^{n/2}$ norm of the bubble weight. For $w_j$, Rellich compactness gives weighted $L^2$ convergence on bounded rescaled balls. Outside such balls, H\"older's inequality bounds the contribution by the $L^{n/2}$ tail norm of $U^{\frac{4}{n-2}}$, which tends to zero.

The center neighborhoods are disjoint. Lower semicontinuity of their rescaled gradient energies, followed by increasing the rescaled radius, gives
$$
\sum_i \frac{4(n-1)}{n-2}\int_{\R^n}|\nabla f_i|^2\dd x
\le1-Q(v_\infty,v_\infty)\le1.
$$
Thus
$$
\limsup_j \frac{n+2}{n-2}\sum_i\int_MU_{\delta_i,P_i}^{\frac{4}{n-2}}v_j^2\dd V_g
\le\frac{n}{n+4}.
$$
This proves \eqref{eq:normal-coercivity}, with any fixed $c<4/(n+4)$ once all scales are sufficiently small. The argument rescaled each center separately and did not constrain their scale ratios.
\end{proof}

\begin{proposition}\label{prop:balanced}
Fix $K$ and a compact family of ordered configurations $P=(P_1,\ldots,P_K)$ with pairwise distances at least $d_0>0$. Suppose $(\psi-R_g)(P_i)\ge h_0>0$ throughout this family. Let $\theta_i(P,\cdot)$ be a smooth nonnegative partition of one, with uniformly bounded derivatives, equal to $\delta_{ij}$ on fixed disjoint balls about $P_j$. There is $\rho_0>0$ such that every $u\in\Neh$ satisfying
\begin{equation}\label{eq:balanced-conditions}
W_1\left(\mu_u,\frac1K\sum_i\delta_{P_i}\right)<\rho_0,
\qquad
\int_M\theta_i(P,x)\dd\mu_u(x)=\frac1K\quad(1\le i\le K)
\end{equation}
has $J(u)>Ke_n$. The assertion is invariant under simultaneous permutations of centers and partition functions, so it also holds for compact unordered families. The neighborhood and all constants in this proposition are chosen after $K$ and the separated center family have been fixed.
\end{proposition}

\begin{proof}
Suppose there are configurations $P^{(j)}$ converging within the compact family and functions $u_j\in\Neh$ with $J(u_j)\le Ke_n$, whose normalized critical measures converge to the corresponding equal-weight atomic measure, and which satisfy the exact equalities in \eqref{eq:balanced-conditions}. Since
$$
Q(u_j,u_j)=\int_Mu_j^{\frac{2n}{n-2}}\dd V_g=nJ(u_j)\le Kne_n,
$$
coercivity and Sobolev embedding bound these integrals above and away from zero. Pass to a limit $D\in(0,Kne_n]$. After a subsequence, let $u_\infty$ be the weak $H^1$ limit. Rellich compactness gives strong $L^2$ and almost-everywhere convergence. On each compact set avoiding the limiting centers, weak convergence of the critical measures makes the $L^{\frac{2n}{n-2}}$ mass tend to zero. Hence $u_\infty$ vanishes there. The omitted finite set has zero volume, so $u_\infty=0$ almost everywhere and $u_j\to0$ in $L^2$.

For a fixed small radius $r$, take a smooth nonnegative quadratic partition $\chi_0^2+\sum_{i=1}^K\chi_i^2=1$, with $\chi_i$ supported in $B(P_i^{(j)},2r)$ and equal to one in $B(P_i^{(j)},r)$, and with $|\nabla\chi_i|\le C/r$ for $0\le i\le K$. The supports with $i\ge1$ are disjoint. The IMS identity gives
$$
\sum_{i=0}^KQ(\chi_i u_j,\chi_i u_j)
=Q(u_j,u_j)+\frac{4(n-1)}{n-2}\int_M\sum_i|\nabla\chi_i|^2u_j^2\dd V_g.
$$
For fixed $r$ the last term vanishes. In a normal-coordinate ball the sharp Euclidean Sobolev inequality \cite{Aubin1976,Talenti1976} is
$$
\frac{4(n-1)}{n-2}\int_M|\nabla(\chi_i u_j)|^2\dd V_g
\ge(1-Cr^2)(ne_n)^{2/n}
\left(\int_M|\chi_i u_j|^{\frac{2n}{n-2}}\dd V_g\right)^{\frac{n-2}{n}}.
$$
The potential terms tend to zero. Since each local critical mass tends to $D/K$, letting first $j\to\infty$ and then $r\to0$ yields
$$
D\ge K(ne_n)^{2/n}(D/K)^{(n-2)/n}.
$$
It follows that $D\ge Kne_n$, so $D=Kne_n$. For each fixed small $r$ and $1\le i\le K$, the local Sobolev bound and the IMS identity give
\begin{align*}
\liminf_{j\to\infty}Q(\chi_i u_j,\chi_i u_j)&\ge ne_n-Cr^2,\\
\limsup_{j\to\infty}Q(\chi_i u_j,\chi_i u_j)&\le ne_n+Cr^2.
\end{align*}
For fixed $K$, the constant $C$ is independent of $r$. The upper bound follows by subtracting the lower bounds for the other $K-1$ localized energies from the total limit $Kne_n$, using $Q(\chi_0u_j,\chi_0u_j)\ge0$. The same identity gives $0\le\limsup_{j\to\infty}Q(\chi_0u_j,\chi_0u_j)\le Cr^2$. After transfer to Euclidean coordinates, the local Sobolev deficits therefore have upper limits bounded by $Cr^2$ as well. Thus these deficits and the remaining localized energy tend to zero when first $j\to\infty$ with $r$ fixed and then $r\downarrow0$. This conclusion uses only the Nehari identity and the exact limiting masses, without a Palais--Smale assumption.

Choose radii $r_j\downarrow0$ by a common diagonal selection for all $K$ centers. More explicitly, for each $\eta>0$, first choose $r$ so that the preceding $Cr^2$ bounds are below $\eta$. For that fixed $r$, take $j$ sufficiently large to control all local mass errors, the IMS error and $r^{-2}\|u_j\|_2^2$ by $\eta$, and all local energy errors by $2\eta$. Letting $\eta\downarrow0$ along a sequence gives radii $r_j$ for which the local masses and gradient energies have the preceding limits, all local Euclidean Sobolev deficits tend to zero, and $r_j^{-2}\|u_j\|_2^2\to0$. In the rescaled coordinate charts the localized functions are therefore Euclidean optimizing sequences. Extend each localized coordinate function by zero to $\R^n$. It belongs to $\dot H^1(\R^n)$, and the metric errors vanish because $r_j\to0$. The theorem of Bianchi and Egnell \cite{BianchiEgnell1991} states that, for every $n\ge3$, the sharp Euclidean Sobolev deficit bounds below a positive dimensional constant times the squared $\dot H^1$ distance to all scalar multiples, translations and dilations of the Euclidean optimizer. Its hypotheses now hold for each of these localized functions. Thus each local function is $o_{\dot H^1}(1)$ from a single positive bubble, with amplitude tending to one. Its scale tends to zero and its center tends to the relevant $P_i^{(j)}$. Its scale divided by the localization radius tends to zero as well; otherwise a definite optimizer tail would remain outside the shrinking localization ball.

Along the chosen diagonal, $Q(\chi_0u_j,\chi_0u_j)\to0$, so coercivity gives $\|\chi_0u_j\|_{H^1(M)}\to0$. Since $0\le1-\sum_{i=1}^K\chi_i\le\chi_0$ and $|\nabla\chi_i|\le C/r_j$, the product rule gives
$$
\left\|u_j-\sum_{i=1}^K\chi_i u_j\right\|_{H^1(M)}
\le C\left(\|\chi_0u_j\|_{H^1(M)}+r_j^{-1}\|u_j\|_{L^2(M)}\right)\longrightarrow0.
$$
Here $C$ is independent of $j$, and the last term tends to zero by the preceding diagonal choice. Denote the center and scale of the $i$th Euclidean optimizing bubble by $a_{i,j}$ and $\delta_{i,j}$ in the normal coordinates at $P_i^{(j)}$, and put $\xi_{i,j}=\exp_{P_i^{(j)}}(a_{i,j})$. Thus $a_{i,j}\to0$ and $\delta_{i,j}\to0$. Identify the tangent spaces with $\R^n$ using an orthonormal frame at $P_i^{(j)}$ and its parallel transport to $\xi_{i,j}$. Smoothness of the exponential maps gives, for every fixed $R<\infty$,
$$
\delta_{i,j}^{-1}\left[\exp_{P_i^{(j)}}^{-1}\!\left(\exp_{\xi_{i,j}}(\delta_{i,j}y)\right)-a_{i,j}\right]
\longrightarrow y\quad\text{in }C^1(B_R(0)).
$$
Consequently the translated Euclidean bubble and the bubble centered at $\xi_{i,j}$ differ by $o(1)$ in gradient norm on each such rescaled ball. On a common fixed coordinate ball, the coordinate changes and their inverses have uniformly bounded differentials. Outside the ball of radius $R\delta_{i,j}$ about $\xi_{i,j}$, the gradient energy of either fixed cut-off representation is therefore bounded by $C\int_{|y|>cR}|\nabla U|^2\dd y+o(1)$, with $c,C>0$ independent of $j$ and $R$; derivatives of the fixed cutoffs contribute $O(\delta_{i,j}^{n-2})$. Letting first $j\to\infty$ and then $R\to\infty$ proves that changing the coordinate center and taking fixed cutoffs supported inside the partition cores changes the approximation by $o_{H^1}(1)$, since both bubble representations have $L^2$ norms tending to zero. Combining this with the local approximations and the preceding estimate gives
\begin{equation}\label{eq:modulated-decomposition}
u_j=\sum_{i=1}^K t_{i,j}U_{\delta_{i,j},\xi_{i,j}}+v_j,
\qquad t_{i,j}\to1,\quad\delta_{i,j}\to0,\quad
\xi_{i,j}-P_i^{(j)}\to0,\quad\|v_j\|_Q\to0.
\end{equation}
Here the center convergence is understood in local coordinates. We next adjust the parameters to make $v_j$ orthogonal to each bubble and its scale and center derivatives. Use amplitude, logarithmic scale, and center displacement divided by the scale as parameter coordinates. Their first and second $H^1$ derivatives are uniformly bounded by rescaling and the cutoff tail estimates. The Gram matrix converges to the Euclidean Gram matrix of
$$
U,\qquad-\frac{n-2}{2}U-x\cdot\nabla U,\qquad\partial_1U,\ldots,\partial_nU.
$$
These functions are independent. The disjoint supports remove cross blocks, while metric and potential terms tend to zero. The Gram inverse is consequently uniformly bounded, even for unequal scales. The derivative of the orthogonality equations differs from this invertible matrix by $O(\|v_j\|_Q)$. The quantitative implicit function theorem gives the required adjustment, preserving all limits and support properties in \eqref{eq:modulated-decomposition}. In the rest of the proof suppress $j$ and use these adjusted parameters.

Since $n\ge6$, Taylor's formula for the power $2n/(n-2)$ has remainder
\begin{equation}\label{eq:power-taylor}
\left|(b+h)_+^{\frac{2n}{n-2}}-b^{\frac{2n}{n-2}}-\frac{2n}{n-2}b^{\frac{n+2}{n-2}}h-\frac{n(n+2)}{(n-2)^2}b^{\frac{4}{n-2}}h^2\right|
\le C|h|^{\frac{2n}{n-2}}\qquad(b\ge0).
\end{equation}
Indeed the second derivative is Hölder with exponent $\frac{4}{n-2}$, including the Lipschitz endpoint $n=6$, and the integral second-order remainder gives the formula. Its first-order version is bounded by $C(b^{\frac{4}{n-2}}h^2+|h|^{\frac{2n}{n-2}})$.

The mass equalities remove the negative amplitude directions. Every adjusted bubble is supported in the core where its own $\theta_i$ is one and the other partition functions vanish. The first-order Taylor estimate, including the region outside all bubble supports, gives
\begin{align*}
\int_M\theta_i u^{\frac{2n}{n-2}}\dd V_g
={}&t_i^{\frac{2n}{n-2}}\int_MU_{\delta_i,\xi_i}^{\frac{2n}{n-2}}\dd V_g\\
&-\frac{2n}{n-2}t_i^{\frac{n+2}{n-2}}\ip{L_\psi U_{\delta_i,\xi_i}-U_{\delta_i,\xi_i}^{\frac{n+2}{n-2}}}{v}
+O(\|v\|_Q^2).
\end{align*}
Indeed $Q(U_{\delta_i,\xi_i},v)=0$ converts the linear mass term to the displayed residual pairing. H\"older's inequality bounds the quadratic mass terms by $C\|v\|_Q^2$; the remaining $\|v\|_{\frac{2n}{n-2}}^{\frac{2n}{n-2}}$ term is absorbed because $\frac{2n}{n-2}>2$ and $v\to0$. All the left-hand sides are equal, while each bubble mass is $ne_n+O(\delta_i^2)$. Since $t_i\to1$, subtracting the identities for $i$ and $l$ yields
\begin{align*}
|t_i-t_l|\le C\biggl(&\sum_j\delta_j^2+\|v\|_Q^2\\
&+\sum_j\|L_\psi U_{\delta_j,\xi_j}-U_{\delta_j,\xi_j}^{\frac{n+2}{n-2}}\|_{H^{-1}}\|v\|_Q\biggr).
\end{align*}
Disjoint supports and the same orthogonality also give
$$
Q(u,u)=\sum_i t_i^2Q(U_{\delta_i,\xi_i},U_{\delta_i,\xi_i})+Q(v,v).
$$
Summing the mass expansions and using $Q(u,u)=\int_Mu^{\frac{2n}{n-2}}\dd V_g$, we obtain
\begin{align*}
\left|ne_n\sum_i(t_i^2-t_i^{\frac{2n}{n-2}})\right|
\le C\biggl(&\sum_i\delta_i^2+\|v\|_Q^2\\
&+\sum_i\|L_\psi U_{\delta_i,\xi_i}-U_{\delta_i,\xi_i}^{\frac{n+2}{n-2}}\|_{H^{-1}}\|v\|_Q\biggr).
\end{align*}
By \eqref{eq:residual-small-square} and Young's inequality both right-hand sides are bounded by $C(\sum_i\delta_i^2+\|v\|_Q^2)$. The pairwise differences control the deviation of each $t_i$ from their mean. Expanding $t^2-t^{\frac{2n}{n-2}}$ at that mean, its derivative tends to $-\frac{4}{n-2}\ne0$; its quadratic error is absorbed because every $t_i\to1$. Consequently
\begin{equation}\label{eq:amplitude-control}
\max_i|t_i-1|\le C\left(\sum_i\delta_i^2+\|v\|_Q^2\right).
\end{equation}

For large $j$, $(\psi-R_g)(\xi_i)\ge h_0/2$. The disjoint supports, \eqref{eq:local-energy} and \eqref{eq:amplitude-control} imply
$$
J\left(\sum_i t_iU_{\delta_i,\xi_i}\right)
\ge Ke_n+c\sum_i\delta_i^2
-C\left[\left(\sum_i\delta_i^2\right)^2+\|v\|_Q^4\right]
+o\left(\sum_i\delta_i^2\right).
$$
The linear amplitude derivative is $O(\delta_i^2)$, and its second derivative is bounded, which accounts for the error shown. Apply \eqref{eq:power-taylor} to add $v$. Orthogonality makes the linear term a sum of residual pairings. Lemma~\ref{lem:normal}, with $t_i\to1$, gives a positive quadratic term. Young's inequality and \eqref{eq:residual-small-square} absorb the residual pairings at a cost $o(\sum_i\delta_i^2)$. We obtain
\begin{align*}
J(u_j)\ge{}&Ke_n+c_1\sum_i\delta_i^2+c_2\|v_j\|_Q^2\\
&-C\left[\left(\sum_i\delta_i^2\right)^2+\|v_j\|_Q^4+\|v_j\|_Q^{\frac{2n}{n-2}}\right]
+o\left(\sum_i\delta_i^2\right).
\end{align*}
Since $\frac{2n}{n-2}>2$, the right side is strictly larger than $Ke_n$ for large $j$. This contradicts the energy bound and proves the proposition.
\end{proof}

\section{Finite cycles of atomic measures}\label{sec:topology}

This section applies to every smooth closed connected $n$-manifold with $n\ge3$. All chains, homology and cohomology groups have coefficients in $\F$. Measures with exactly $K$ distinct support points form an open smooth stratum of $B_K(M)$: locally order the centers in disjoint neighborhoods and use their positions and positive weights as coordinates.

\begin{proposition}\label{prop:source}
There is a continuous map $f:M\to\R^{n-1}$ and a constant $L_0<\infty$ with the following properties. For every $K\ge1$, the compact space
\begin{equation}\label{eq:source-space}
Z_K=\{\nu\in B_K(M):f\hbox{ is constant on }\supp\nu\}
\end{equation}
has a finite singular cycle $z_K$ of dimension $2K+n-2$. On the open stratum with $K$ distinct centers avoiding a fixed $(n-2)$-dimensional skeleton, this cycle has local fundamental coefficient one. The support of every $\nu\in Z_K$ is contained in a rectifiable closed curve of length at most $L_0$, independently of $K$.
\end{proposition}

\begin{proof}
Take a finite smooth compatible triangulation $\tau:|T|\to M$: each closed simplex is smoothly embedded with injective differential, and $|T|$ is a piecewise linear manifold. We use the compatible triangulation theorem in this precise form; a proof and its relation to Whitehead's theorem are given in \cite{Lurie2009}. Assign vectors of $\R^{n-1}$ to the finitely many vertices so that every $(n-1)$-simplex has affinely independent image vertices. These are finitely many nonvanishing polynomial conditions. Each determinant polynomial is nonzero, so their product is nonzero and such an assignment exists. Affine extension on simplices gives $F:|T|\to\R^{n-1}$; set $f=F\circ\tau^{-1}$.

The restriction to a $k$-face has rank $k$ when $k\le n-1$, and the restriction to an $n$-simplex has rank $n-1$. A fiber in an $n$-simplex is a line segment, a point or the empty set. Put $S=\tau(|T^{(n-2)}|)\subset M$. Off $S$, the map is locally a product projection. Inside an $n$-simplex this follows by adding a complementary affine coordinate. Across an $(n-1)$-face choose tangential coordinates $x'$ and signed normal coordinate $t$. The two affine expressions are
$$
F(x',t)=Ax'+b_+t+c\quad(t\ge0),\qquad
F(x',t)=Ax'+b_-t+c\quad(t\le0),
$$
with the same invertible matrix $A$. The map $(x',t)\mapsto(F(x',t),t)$ has a piecewise affine inverse, so its fiber is locally a full interval rather than an interval with boundary.

The condition in \eqref{eq:source-space} is closed, since it is equivalent to
$$
\int_M|f|^2\dd\nu-\left|\int_Mf\dd\nu\right|^2=0.
$$
Hence $Z_K$ is compact. To obtain a finite triangulation, work first on a Euclidean realization of $|T|$. The moments of a measure through degree $2K$ embed $B_K(|T|)$ into a finite-dimensional Euclidean space. For injectivity, if a measure is supported at $x_1,\ldots,x_m$, $m\le K$, the nonnegative polynomial $\prod_i|x-x_i|^2$ forces every measure with the same moments to have support in this set. The polynomials $\prod_{j\ne i}|x-x_j|^2$ then determine each weight. Compactness makes the injection an embedding. The moment image, the common-fiber condition and all specified face and support conditions are semialgebraic: each is the projection of finitely many polynomial conditions on centers and weights. The finite compatible semialgebraic triangulation theorem, in the form \cite[Definition 3.11 and Theorem 3.12]{Czapla2012}, therefore applies. Push this triangulation forward by $\tau$ on measures.

Let $D_K\subset Z_K$ consist of measures with at most $K-1$ support points together with $K$-point measures having a center in $S$. This set is closed: if a marked skeleton center loses its positive weight, the limit belongs to the lower-support part. Outside $D_K$, order the centers in disjoint local neighborhoods. The product coordinates for $f$ give a common variable $y\in\R^{n-1}$, one fiber coordinate for each center, and $K-1$ positive weight coordinates. Thus $Z_K\setminus D_K$ is a manifold without boundary of dimension $2K+n-2$; the unordered description gives the same local charts because the centers are distinct.

We check all exceptional dimensions. If $m$ distinct centers lie in prescribed relative open faces of dimensions $k_1,\ldots,k_m$, and their affine image intersection has dimension $b$, the center-parameter dimension is at most
$$
b+\sum_i\max\{k_i-(n-1),0\}.
$$
There are another $m-1$ weight parameters. For $m\le K-1$, this gives at most $(n-1)+m+(m-1)\le 2K+n-4$. For $m=K$ with at least one center on the $(n-2)$ skeleton, one has $b\le n-2$ and at most $K-1$ fiber directions, again giving at most $2K+n-4$. Multiple skeleton centers, collisions and weight losses satisfy these same inequalities. Thus $\dim D_K\le 2K+n-4$.

Choose the finite triangulation compatible with $D_K$, order its vertices, and sum all its simplices of dimension $2K+n-2$ with coefficient one. The relative interior of any incident $(2K+n-3)$-face is outside $D_K$. It is a manifold point and meets exactly two top-dimensional simplices. Their boundaries cancel modulo two, giving the finite closed singular chain $z_K$. At every point of the manifold stratum its local fundamental coefficient is one. Lower-dimensional components do not create a top-dimensional boundary.

Finally, in each original $n$-simplex the fiber segment has length bounded by a fixed Lipschitz constant for $\tau$ times that simplex's diameter. Summing over the finitely many simplices bounds the total fiber length independently of the fiber and of $K$. Its number of components is also uniformly finite. Join these components by finitely many minimizing geodesics, double the edges of the resulting finite abstract graph, and traverse them. This produces a closed rectifiable curve containing the entire fiber with a uniform length bound $L_0$. No continuous choice of this auxiliary curve is required.
\end{proof}

\begin{proposition}\label{prop:tree}
Let $K=2^r$, $r\ge1$, and let $O\subset M$ be nonempty and open. Inside the open $K$-point stratum of $B_K(M)$ there is a compact smooth embedded manifold $C_K$, consisting of equal-weight measures supported in $O$, such that
$$
\dim C_K=(n-1)(K-1),\qquad\operatorname{codim} C_K=2K+n-2.
$$
Its centers are uniformly separated for this fixed $K$. Its mod-two normal Thom class $\tau_C$, regarded by excision as a class of $H^{2K+n-2}(B_K(M),B_K(M)\setminus C_K)$, satisfies
\begin{equation}\label{eq:source-pairing}
\ip{\overline{\tau_C}}{[z_K]}=1,
\end{equation}
where the bar denotes passage to absolute cohomology.
\end{proposition}

\begin{proof}
Choose a chart in $O$ inside an original $n$-simplex, with coordinates $(y,t)$ such that $y=f-f(P_0)\in\R^{n-1}$. Use the complete rooted binary tree of depth $r$, with $K=2^r$ leaves. Thus the two child clusters at every internal vertex have the same number of leaves. For each internal vertex $I$ choose $v_I\in S^{n-1}\subset\R^n$. Choose positive radii $a_j$, indexed by depth, so small that all the following points lie in the chart and
$$
\sum_{l>j}a_l<a_j/10.
$$
For a leaf $\eta$ put
\begin{equation}\label{eq:tree-leaf}
x_\eta=\sum_{I\text{ ancestor of }\eta}
\sigma_{I,\eta}a_{\operatorname{depth}I}v_I,
\end{equation}
where $\sigma_{I,\eta}$ is $+1$ or $-1$ according to its child branch. Take the equal-weight measure of these chart points. If two leaves first differ at depth $j$, their distance in coordinates is at least $2a_j-2\sum_{l>j}a_l>0$.

The group of binary child exchanges has order $2^{K-1}$. It acts on $(S^{n-1})^{K-1}$ by reversing the direction at an exchanged vertex and permuting its descendant directions. This action is free: at a highest exchanged vertex no ancestor exchange occurs, and a fixed point would require $v_I=-v_I$. Its quotient is a compact smooth manifold of dimension $(n-1)(K-1)$.

The map from this quotient to unordered configurations is an embedding. The radius inequalities distinguish the two root clusters by a distance threshold and then distinguish the two clusters at every subsequent vertex. Within each child cluster every descendant contribution has arithmetic mean zero: at each descendant vertex its two signs occur equally often. The difference of the two child-cluster means is therefore exactly $2a_{\operatorname{depth}I}v_I$. It recovers $v_I$ up to child exchange. With local labels fixed, this recovery is smooth and provides a local inverse. Thus the configuration map is an injective immersion, and compactness makes it an embedding. Denote its image in the open $K$-point stratum by $C_K$. That stratum has dimension $nK+(K-1)=(n+1)K-1$, so the stated codimension follows.

The following intersection calculation is in the unordered configuration space, so no permutation multiplicity is introduced. The intersection of $C_K$ and $Z_K$ consists of one point. Indeed the common-fiber condition makes all leaf $y$ coordinates equal. Their mean is zero by \eqref{eq:tree-leaf}, so they all vanish. Differences of child-cluster means then force every $v_I$ to be one of the two unit vectors on the $t$ axis. Conversely those choices satisfy the condition. All $2^{K-1}$ sign choices form one orbit of the child-exchange group.

This intersection is transverse. At these axis directions the tangent variables are the first $n-1$ components of the $v_I$. Their map to centered leaf arrays $(\delta y_\eta)$ is a linear isomorphism: for each vertex take the difference of the two child-cluster means and divide by $2a_{\operatorname{depth}I}$ to recover its tangent variable. The common-fiber equations have exactly these $(n-1)(K-1)$ independent center components; all weight directions belong to the source tangent space. The single intersection is in the smooth source stratum, where $z_K$ has local coefficient one.

The compact embedded $C_K$ has a tubular neighborhood entirely in the open $K$-point stratum. Its normal bundle has rank $2K+n-2$ and a mod-two Thom class, independently of orientations of the manifold or bundle. Excision gives the indicated class on $B_K(M)$, and transverse local intersection with the finite cycle computes its pairing as one. This proves \eqref{eq:source-pairing}.
\end{proof}

\section{Positive test functions and their energy homotopy}\label{sec:tests}

In this section $n\ge6$. Fix a cutoff as in \eqref{eq:cutoff-bubble}, equal to one for $r\le r_0$ and zero for $r\ge2r_0$, with $2r_0$ below the injectivity radius. Set
\begin{equation}\label{eq:projected-bubble}
\phi_{\delta,P}=L_\psi^{-1}(U_{\delta,P}^{\frac{n+2}{n-2}}).
\end{equation}
These functions are smooth and strictly positive. For $\nu\in\Prob(M)$ and $\delta>0$, define
\begin{equation}\label{eq:test-map}
G_\delta(\nu)=\Pi\left[\left(\int_M\phi_{\delta,P}^{\frac{2n}{n-2}}\dd\nu(P)\right)^{\frac{n-2}{2n}}\right].
\end{equation}
In the integral $P$ is the source variable and the resulting function depends on $x\in M$. The definition depends on the measure itself, so it remains continuous when atoms merge or a weight tends to zero. For an atomic measure the critical density before normalization is exactly the weighted sum of the kernels $\phi_{\delta,P}^{\frac{2n}{n-2}}$. This identity will control both the energy and the probability measure of the test function.

\begin{lemma}\label{lem:green-bubbles}
There are positive constants $c,C$ such that, uniformly in $P,x$ and small $\delta$,
\begin{equation}\label{eq:projected-comparison}
c\frac{\delta^{(n-2)/2}}{(\delta^2+d_g(P,x)^2)^{(n-2)/2}}
\le\phi_{\delta,P}(x)
\le C\frac{\delta^{(n-2)/2}}{(\delta^2+d_g(P,x)^2)^{(n-2)/2}}.
\end{equation}
On the core, the ratio of $\phi_{\delta,P}$ to the uncut bubble differs from one by at most $C(\delta^2+d_g(P,x)^2)$. Moreover,
\begin{equation}\label{eq:kernel-masses}
\begin{split}
\int_M\left(\frac{U_{\delta,P}^{\frac{n+2}{n-2}}}{\phi_{\delta,P}}\right)^{n/2}\dd V_g
&=ne_n+O(\delta^2),\\
\int_M\phi_{\delta,P}^{\frac{2n}{n-2}}\dd V_g&=ne_n+O(\delta^2),\\
\int_Md_g(P,x)\phi_{\delta,P}(x)^{\frac{2n}{n-2}}\dd V_g(x)&\le C\delta.
\end{split}
\end{equation}
\end{lemma}

\begin{proof}
The coercive inverse has a Green kernel $G_{L_\psi}(x,y)$. The needed local asymptotic can be obtained as follows. Apply $L_\psi$ to the cut-off kernel $[4(n-1)|S^{n-1}|]^{-1}d_g(x,y)^{2-n}$. Its error away from the Dirac mass is $O(d_g(x,y)^{2-n})$, by the radial normal-coordinate expansion. Successive convolution corrections improve the singular order by two: splitting each convolution into balls about its two poles and the complement gives the bounds $O(d_g(x,y)^{4-n})$, then $O(d_g(x,y)^{6-n})$, with a logarithm at a zero exponent. After finitely many such corrections the remaining kernel is continuous; taking further corrections makes it H\"older continuous. Apply the coercive inverse to that remainder. Elliptic estimates make this last correction bounded. This constructs the inverse kernel and proves its leading fundamental-solution asymptotic. Positivity follows by approximating a nonnegative Dirac mass and using the weak maximum principle from Lemma~\ref{lem:flow}; the strong maximum principle gives strict positivity off the diagonal. Compactness away from the diagonal then gives
$$
cd_g(x,y)^{2-n}\le G_{L_\psi}(x,y)\le Cd_g(x,y)^{2-n}.
$$
These statements concern the fixed smooth uniformly elliptic operator $L_\psi$; the constants may depend on its coefficients.

Let $r=d_g(P,x)$. The source in \eqref{eq:projected-bubble} has mass comparable to $\delta^{(n-2)/2}$ in $B(P,\delta/2)$. The Green lower bound, integrated on a smaller such ball if $r\le2\delta$, gives the lower estimate in \eqref{eq:projected-comparison}. For the upper estimate when $r\le2\delta$, rescale the core and sum radial annuli; the convolution of the source decay with $|x-y|^{2-n}$ is $O(\delta^{-(n-2)/2})$. When $r\ge2\delta$, split the integral into $B(P,r/2)$, $B(x,r/2)$ and their complement. The first part is at most $Cr^{2-n}\delta^{(n-2)/2}$. On the second, the source is bounded by $C\delta^{(n+2)/2}r^{-n-2}$ and the kernel integral by $Cr^2$. On the complement the kernel is at most $Cr^{2-n}$ and the total source mass is $O(\delta^{(n-2)/2})$. This proves the upper estimate.

The residual in Lemma~\ref{lem:local-expansion} is pointwise bounded by a constant times the uncut bubble, including on the fixed annulus. The Green representation of the difference therefore gives
\begin{equation}\label{eq:projection-difference}
|U_{\delta,P}(x)-\phi_{\delta,P}(x)|
\le C\frac{\delta^{(n-2)/2}}{(\delta^2+r^2)^{(n-4)/2}}.
\end{equation}
Here is the needed convolution estimate. For $r\le2\delta$, the core and the annuli are bounded by a constant times
$$
\delta^{(6-n)/2}
+\delta^{(n-2)/2}\int_\delta^{\diam M}\rho^{3-n}\dd\rho
\le C\delta^{(6-n)/2}.
$$
For $r\ge2\delta$, the same three-region split bounds the first two contributions by $C\delta^{(n-2)/2}r^{4-n}$. The outer radial annuli contribute at most $C\delta^{(n-2)/2}\int_r^{\diam M}\rho^{3-n}\dd\rho$, which has the same bound because $n>4$. This proves \eqref{eq:projection-difference} and the asserted core ratio estimate.

The ratio is bounded above and away from zero by \eqref{eq:projected-comparison}. Thus any of the fixed powers in \eqref{eq:kernel-masses} has a Lipschitz dependence on that ratio. In the core its difference from the uncut critical density is bounded by $C(\delta^2+r^2)$ times that density. Its integrated error is $O(\delta^2)$. Outside the core, \eqref{eq:projected-comparison} bounds the relevant densities by $C\delta^n r^{-2n}$; their total mass is $O(\delta^n)$. Normal-coordinate integration proves the first two formulas in \eqref{eq:kernel-masses}. The same pointwise bound and the finite Euclidean first moment of $U^{\frac{2n}{n-2}}$ give the last formula.
\end{proof}

\begin{lemma}\label{lem:envelope}
If $\nu=\sum_{i=1}^m\alpha_i\delta_{P_i}$ has distinct support points and positive weights, then
\begin{equation}\label{eq:envelope}
J(G_\delta(\nu))\le\frac1n\int_M\max_{1\le i\le m}
\left(\frac{U_{\delta,P_i}^{\frac{n+2}{n-2}}}{\phi_{\delta,P_i}}\right)^{n/2}\dd V_g.
\end{equation}
For every collection of disjoint index pairs with $d_g(P_i,P_j)+\delta$ sufficiently small compared with $r_0$, the right side is at most
\begin{equation}\label{eq:overlap-energy}
me_n+Cm\delta^2-c\sum_{(i,j)}
\left(\frac{\delta}{\delta+d_g(P_i,P_j)}\right)^n.
\end{equation}
The constants are independent of $m$, the weights, the centers and the chosen disjoint pairs.
\end{lemma}

\begin{proof}
We estimate the entire ray before projecting to $\Neh$, retaining the loss caused by overlapping kernels. The function $F(t)=(\sum_i\alpha_i t_i^{\frac{2n}{n-2}})^{\frac{n-2}{2n}}$ is convex and positively homogeneous of degree one on the positive orthant. Write $v=F(\phi_{\delta,P_1},\ldots,\phi_{\delta,P_m})$. Its Hessian is positive semidefinite, its first derivatives are nonnegative, and $\sum_iF_i\phi_{\delta,P_i}=v$. The chain rule gives
\begin{align*}
L_\psi v
&=\sum_iF_iU_{\delta,P_i}^{\frac{n+2}{n-2}}
-\frac{4(n-1)}{n-2}\sum_{i,j}F_{ij}\ip{\nabla\phi_{\delta,P_i}}{\nabla\phi_{\delta,P_j}}\\
&\le\left(\max_i\frac{U_{\delta,P_i}^{\frac{n+2}{n-2}}}{\phi_{\delta,P_i}}\right)v.
\end{align*}
The term with $\psi$ cancels exactly by homogeneity, so no sign assumption on $\psi$ is involved. Multiply by $v$ and integrate. For $b\ge0$ and $z\ge0$,
$$
\frac12bz^2-\frac{n-2}{2n}z^{\frac{2n}{n-2}}\le\frac1n b^{n/2},
$$
as follows by maximizing in $z$. Apply this inequality pointwise to $z=tv$ and then maximize over $t>0$ in \eqref{eq:ray-energy}. This proves \eqref{eq:envelope}.

For nonnegative numbers $f_1,\ldots,f_m$ and disjoint pairs,
$$
\sum_i f_i-\max_i f_i\ge\sum_{(i,j)}\min(f_i,f_j).
$$
Indeed each paired sum is its maximum plus its minimum; the sum of the pair maxima and unpaired terms dominates the global maximum. To estimate an overlap integral, put $d=d_g(P_i,P_j)$ and integrate on $B(P_i,(\delta+d)/4)$. Both centers' cutoffs are one there, provided $\delta+d$ is sufficiently small. By the upper comparison for $\phi$ in \eqref{eq:projected-comparison}, each density in \eqref{eq:envelope} is bounded below on this ball by $c\delta^n/(\delta+d)^{2n}$. Its volume is at least $c(\delta+d)^n$. Hence the overlap integral is at least $c[\delta/(\delta+d)]^n$. Sum over disjoint pairs and use \eqref{eq:kernel-masses}. This proves \eqref{eq:overlap-energy}, including arbitrarily close centers.
\end{proof}

\begin{proposition}\label{prop:cylinder}
Fix $L_0<\infty$. There are constants $M_0,c_*,B_*>0$ and $K_0$, depending on $n,M,g,\psi,L_0$ and the cutoff, with the following property. Let $K\ge K_0$ be even, and suppose the support of $\nu\in B_K(M)$ is contained in a rectifiable closed curve of length at most $L_0$. With
\begin{equation}\label{eq:cylinder-scale}
\delta_K=M_0K^{-n/(n-2)},
\end{equation}
one has
\begin{equation}\label{eq:cylinder-energy}
J(G_{\delta_K}(\nu))\le Ke_n-c_*K^{-\frac{n+2}{n-2}},
\qquad
\sup_{0<\delta\le\delta_K}J(G_\delta(\nu))\le Ke_n+B_*K^{-\frac{n+2}{n-2}}.
\end{equation}
For fixed $\delta>0$, $G_\delta:\Prob(M)\to\Neh$ is continuous, and
\begin{equation}\label{eq:mass-approximation}
\sup_{\nu\in B_K(M)}W_1(\mu_{G_\delta(\nu)},\nu)\le C\delta
\end{equation}
with $C$ independent of $K$.
\end{proposition}

\begin{proof}
First assume that the actual support has $K$ points. Choose one preimage of each point on an arclength parametrization of the containing closed curve, order those parameters cyclically, and pair successive parameters. Their selected $K/2$ gaps have total length at most $L_0$. At least $K/4$ of the pairs consequently have geodesic distance at most $4L_0/K$. The choice of curve and pairs is only used to prove a pointwise inequality for the given configuration.

For large $K$ and $K\delta\le1$, these pairs satisfy the core condition of Lemma~\ref{lem:envelope}, and
$$
J(G_\delta(\nu))\le Ke_n+CK\delta^2-cK^{n+1}\delta^n.
$$
At the scale \eqref{eq:cylinder-scale} the two corrections are
$$
K\delta_K^2=M_0^2K^{-\frac{n+2}{n-2}},\qquad
K^{n+1}\delta_K^n=M_0^nK^{-\frac{n+2}{n-2}},\qquad
K\delta_K=M_0K^{-2/(n-2)}\to0.
$$
Choose $M_0$ first so that $cM_0^n-CM_0^2>0$, and then choose $K_0$ so large that every pair used above lies in the cutoff cores and $K\delta_K\le1$. These choices make the terminal correction strictly negative. Discarding the negative overlap term gives the asserted upper bound throughout $0<\delta\le\delta_K$. The constants in this argument are independent of $K$ and of all weights. If the actual support has $m\le K-1$ points, the bound without overlaps is
$$
J(G_\delta(\nu))\le(K-1)e_n+CK\delta_K^2<Ke_n-e_n/2
$$
for large $K$. Increasing $K_0$ and decreasing $c_*$ if necessary proves \eqref{eq:cylinder-energy} for every support cardinality. The estimates do not require a positive lower bound on the weights; thus they include every zero-weight face and every collision.

For $\delta>0$, the kernel $(P,x)\mapsto\phi_{\delta,P}(x)^{\frac{2n}{n-2}}$ is smooth and strictly positive. Weak convergence of measures gives uniform convergence of its integrals and of each $x$ derivative, by compactness and a finite-net argument. Raising to the power $(n-2)/(2n)$ and applying \eqref{eq:projection} proves continuity of $G_\delta$, jointly also in $\delta>0$.

For an atomic measure, the critical density before normalization is exactly $\sum_i\alpha_i\phi_{\delta,P_i}^{\frac{2n}{n-2}}$. By \eqref{eq:kernel-masses} its normalized mixing weights differ from $\alpha_i$ by total variation at most $C\delta^2$, uniformly in the number of atoms. Each normalized kernel is within $C\delta$ of its center in $W_1$. Couple each kernel to that center and then correct the mixing weights. The total cost is at most $C\delta+\diam(M)C\delta^2$, proving \eqref{eq:mass-approximation}. The Nehari scalar does not change the normalized critical measure.
\end{proof}

\section{The relative energy class and the existence proof}\label{sec:detector}

The intersection in \eqref{eq:source-pairing} is initially a statement about $B_K(M)$. We now construct a cohomology class in a relative pair of full Nehari sublevels which detects the same finite cycle. The exact local mass estimate supplies the support condition needed for this transfer.

\begin{proposition}\label{prop:target-class}
Let $n\ge6$ and suppose that $\psi-R_g>0$ on a nonempty open set. Fix a dyadic $K\ge2$ and choose $C_K$ there as in Proposition~\ref{prop:tree}. Assume that \eqref{eq:main} has no positive solution. For every sufficiently small $\varepsilon>0$ and then all sufficiently small $\delta>0$,
\begin{equation}\label{eq:target-nonzero}
[(G_\delta)_*z_K]\ne0
\quad\hbox{in}\quad
H_{2K+n-2}(\Neh^{Ke_n+\varepsilon},\Neh^{Ke_n-\varepsilon};\F).
\end{equation}
\end{proposition}

\begin{proof}
Fix $K$ throughout the construction. Choose a compact neighborhood of the center configurations of $C_K$, still separated and contained in $\{\psi-R_g>0\}$. On this family choose disjoint radial cutoffs $\chi_i(P,x)$, equal to one in smaller center balls, and set
\begin{equation}\label{eq:partition-selector}
\theta_i(P,x)=\chi_i(P,x)+\frac{1-\sum_j\chi_j(P,x)}K.
\end{equation}
These are a nonnegative partition of one with the core property of Proposition~\ref{prop:balanced}. Their dependence is permutation equivariant. Let $\rho_0$ be the neighborhood size furnished by that proposition.

There is a continuous local center and weight selector on a neighborhood of $C_K$ in $\Prob(M)$. To construct it in an ordered local chart, solve
\begin{equation}\label{eq:selector-centers}
\int_M\chi_i(P,x)\exp_{P_i}^{-1}(x)\dd\nu(x)=0,
\qquad i=1,\ldots,K.
\end{equation}
Use smooth local frames and extend each integrand by zero off its cutoff support. At a $K$-atom measure $\sum_i w_i\delta_{Q_i}$ in the indicated separated neighborhood, the derivative in $P$ at $P_i=Q_i$ is block diagonal with blocks $-w_i\operatorname{Id}$. Terms from differentiating the cutoff or frame vanish at the atom. The integral kernels and their first center derivatives are uniformly continuous, so integration against $\nu$ is a continuous parameter for the finite-dimensional implicit function theorem. Near the compact $C_K$, the weights are bounded away from zero. Uniform local inverses and unique close matching glue the roots, up to the required permutation, into a continuous selector.

Write these centers as $P_i(\nu)$, define the exact weights
$$
w_i(\nu)=\int_M\theta_i(P(\nu),x)\dd\nu(x),
\qquad
R_{\rm loc}(\nu)=\sum_iw_i(\nu)\delta_{P_i(\nu)}.
$$
After shrinking the neighborhood of $C_K$, these weights are positive and sum to one. The selector fixes every measure in $B_K(M)$ sufficiently near $C_K$, not just the equal-weight ones. Indeed such a measure must have positive mass in each of the $K$ separated core neighborhoods. With at most $K$ atoms, there is exactly one in each neighborhood and none elsewhere. Equations~\eqref{eq:selector-centers} recover its centers, and \eqref{eq:partition-selector} recovers its weights.

Choose $R>0$ so that $R_{\rm loc}$ is defined on the closed $4R$-neighborhood of $C_K$, its selected configurations stay in the chosen compact family, and
$$
W_1(\nu,R_{\rm loc}(\nu))<\rho_0
$$
there. This follows by continuity and compactness, since $R_{\rm loc}$ fixes $C_K$. The compact set
$$
\left\{\nu:R\le\dist_{W_1}(\nu,C_K)\le2R,
\ R_{\rm loc}(\nu)\in C_K\right\}
$$
is disjoint from $B_K(M)$ by the fixing property. If nonempty, its distance to the compact $B_K(M)$ is positive. Choose $\rho_2>0$ smaller than that distance, and put
$$
\Um=\{\nu:\dist_{W_1}(\nu,B_K(M))<\rho_2\}.
$$
The entire annular inverse image is now excluded from $\Um$. Therefore
\begin{equation}\label{eq:closed-support}
S_C=\{\nu\in\Um:\dist_{W_1}(\nu,C_K)<2R,
\ R_{\rm loc}(\nu)\in C_K\}
\end{equation}
is closed in $\Um$ and is contained in the $R$-neighborhood of $C_K$. To check closedness, a limit in $\Um$ retains the selector condition. A possible limit in the annulus from $R$ to $2R$ was excluded by the choice of $\rho_2$. Also $S_C\cap B_K(M)=C_K$.

If $\mu_u\in S_C$, its selected centers give an equal-weight configuration in $C_K$. Thus the exact weights in \eqref{eq:partition-selector} are $1/K$, and Proposition~\ref{prop:balanced} gives
\begin{equation}\label{eq:support-avoidance}
\mu_u\in S_C\quad\Longrightarrow\quad J(u)>Ke_n.
\end{equation}
The full condition $R_{\rm loc}(\nu)\in C_K$ in \eqref{eq:closed-support} is used here. Equality of the weights alone would not constrain the selected centers to $C_K$.

On the subset $\{\nu\in\Um:\dist_{W_1}(\nu,C_K)<2R\}$, pull back $\tau_C$ by $R_{\rm loc}$. The resulting class is relative to the complement of $S_C$ in this subset. Since $S_C$ is contained in the smaller $R$-neighborhood of $C_K$, excision extends it uniquely to a class
$$
\alpha\in H^{2K+n-2}(\Um,\Um\setminus S_C).
$$
Its restriction to $(B_K(M),B_K(M)\setminus C_K)$ is $\tau_C$: near $C_K$ the selector is the identity on $B_K(M)$, and both classes are supported on $C_K$ after restriction to $B_K(M)$.

Let $X=\Neh^{Ke_n+\varepsilon}$ and $A=\Neh^{Ke_n-\varepsilon}$, and set
$$
\mathscr O=X\cap\mu^{-1}(\Um),\qquad Z=A\cup\mathscr O.
$$
By \eqref{eq:support-avoidance}, the map $\mu$ sends $(\mathscr O,A\cap\mathscr O)$ to $(\Um,\Um\setminus S_C)$. Pulling back $\alpha$ and applying open-cover excision gives
$$
\beta_Z\in H^{2K+n-2}(Z,A).
$$
Here both $A$ and $\mathscr O$ are open in $Z$. Concretely, represent $\alpha$ by a relative cocycle which is zero on singular chains in $\Um\setminus S_C$. Its pullback is zero on $A\cap\mathscr O$, and therefore agrees with the zero cocycle on $A$. It defines a cocycle on the singular chain subcomplex subordinate to the cover $\{A,\mathscr O\}$. Barycentric subdivision gives its class on the full relative chain complex. Thus the extension has no unaccounted boundary along the overlap.

Choose $0<\rho_1<\rho_2$. Lemma~\ref{lem:envelope} gives $J(G_\delta(\nu))\le Ke_n+CK\delta^2$ on all of $B_K(M)$. By \eqref{eq:mass-approximation}, for small $\delta$ the whole image of the map \eqref{eq:test-map} on $B_K(M)$ lies in $\mathscr O$, with its measures in the inner tube. The probability-measure homotopy
$$
H_t(\nu)=(1-t)\nu+t\mu_{G_\delta(\nu)}
$$
stays in $\Um$, since $W_1(H_t(\nu),\nu)\le tW_1(\mu_{G_\delta(\nu)},\nu)<\rho_1$. Use the absolute image of $\alpha$ to compute the pairing; the homotopy need not be a map of relative pairs. Equation~\eqref{eq:source-pairing} gives
\begin{equation}\label{eq:localized-pairing}
\ip{\beta_Z}{(G_\delta)_*[z_K]}
=\ip{\overline\alpha}{(\mu\circ G_\delta)_*[z_K]}
=\ip{\overline{\tau_C}}{[z_K]}=1.
\end{equation}
The source chain is absolutely closed, so this calculation also computes its relative pairing.

Under the no-solution assumption, Proposition~\ref{prop:localization}, for sufficiently small $\varepsilon$, gives a time $T<\infty$ such that $\Phi_T:(X,A)\to(Z,A)$. Since the source starts in the inner measure tube, its whole flow homotopy lies in $Z$. Therefore $\Phi_T\circ G_\delta$ and $G_\delta$ give the same relative source class in $(Z,A)$. By naturality of the cohomology--homology pairing,
$$
\ip{\Phi_T^*\beta_Z}{(G_\delta)_*[z_K]}
=\ip{\beta_Z}{(\Phi_T\circ G_\delta)_*[z_K]}
=1.
$$
The final equality uses the flow homotopy in $Z$ and \eqref{eq:localized-pairing}. Thus the cycle is nonzero in the pair $(X,A)$, as asserted in \eqref{eq:target-nonzero}.
\end{proof}

\begin{proof}[Proof of Theorem~\ref{thm:main}]
Necessity was proved in the Introduction. Assume $\lambda_1(L_\psi)>0$. For $3\le n\le5$ apply \cite[Theorem 2, p.~41]{BahriBrezis1996} to the coercive operator $-\Delta_g+\frac{(n-2)\psi}{4(n-1)}$. If $v>0$ is the resulting solution, then $u=\left(\frac{4(n-1)}{n-2}\right)^{(n-2)/4}v$ satisfies \eqref{eq:main}, and Lemma~\ref{lem:regularity} gives the asserted smoothness.

Let now $n\ge6$. If $(\psi-R_g)(P)<0$ at some point, \eqref{eq:local-energy} gives $J(\Pi(U_{\delta,P}))<e_n$ for small $\delta$. If there were no positive solution, Lemmas~\ref{lem:quantization} and \ref{lem:flow} would bound $\|V\|_Q$ below on the energy cap below this value: every contrary sequence would have a positive Palais--Smale level below $e_n$. Equation~\eqref{eq:flow-energy} would then force energy to decrease at a fixed positive rate for all forward time, contradicting its positive lower bound. Hence a solution exists in this case.

If $\psi=R_g$ identically, the Yamabe theorem \cite{LeeParker1987} supplies a positive smooth $w$ with $L_\psi w=\lambda w^{\frac{n+2}{n-2}}$. Coercivity gives $\lambda=Q(w,w)/\int_Mw^{\frac{2n}{n-2}}\dd V_g>0$. Then $u=\lambda^{(n-2)/4}w$ solves \eqref{eq:main}.

It remains to treat $\psi-R_g\ge0$ and not identically zero. Suppose for contradiction that no positive solution exists. Fix $f$ and $L_0$ from Proposition~\ref{prop:source}. Choose $M_0$ and then a sufficiently large dyadic $K$ as in Proposition~\ref{prop:cylinder}, with
$$
Ke_n+B_*K^{-\frac{n+2}{n-2}}<b<(K+1)e_n.
$$
For this fixed $K$, choose $C_K$ in the nonempty open set $\{\psi-R_g>0\}$. Next fix its compact separated center family, the local mass estimate, the selector and the two measure tubes in Proposition~\ref{prop:target-class}. The latter proposition and the localization time can depend on all these choices. Only afterwards take
$$
0<\varepsilon<\min\{\varepsilon_0,c_*K^{-\frac{n+2}{n-2}}/2,(b-Ke_n)/4\},
$$
where $\varepsilon_0$ is small enough for the localization and target class. Finally choose $0<\delta_0<\delta_K$ so small that the source has measures in the inner tube and energy below $Ke_n+\varepsilon/4$.

Set $A=\Neh^{Ke_n-\varepsilon}$, $X=\Neh^{Ke_n+\varepsilon}$ and $Y=\Neh^b$. Proposition~\ref{prop:target-class} makes $(G_{\delta_0})_*z_K$ nonzero in $H_{2K+n-2}(X,A)$. On the other hand, Proposition~\ref{prop:cylinder} gives the homotopy $G_\delta$ on $Z_K$, for $\delta_0\le\delta\le\delta_K$, entirely in $Y$, with terminal image in $A$, since $J(G_{\delta_K}(\nu))\le Ke_n-c_*K^{-\frac{n+2}{n-2}}<Ke_n-\varepsilon$ for every $\nu\in Z_K$. Applied to the finite closed chain $z_K$, its singular prism $P$ satisfies
$$
\partial P=(G_{\delta_0})_*z_K+(G_{\delta_K})_*z_K.
$$
Thus the same source class is zero in $H_{2K+n-2}(Y,A)$.

The closed interval $[Ke_n+\varepsilon/2,b]$ contains no integer bubble level. Lemma~\ref{lem:gap} therefore gives one finite time sending $Y$ into $\Neh^{Ke_n+\varepsilon/2}\subset X$ and preserving $A$. This deformation crosses no concentration level. Apply this single continuous map to every simplex of $P$. Its boundary is the flowed source plus a chain in $A$. The original source has energy below $Ke_n+\varepsilon/4$, so its full homotopy to the flowed source stays in $X$. Adding that source prism gives a relative filling of the original $(G_{\delta_0})_*z_K$ in $(X,A)$. This contradicts \eqref{eq:target-nonzero}.

The no-solution assumption has led to a contradiction in every remaining case. Hence there is a nonzero nonnegative weak solution, and Lemma~\ref{lem:regularity} makes it smooth and strictly positive.
\end{proof}

\vspace{0.5cm}

%\textbf{Funding} No funding.

\textbf{Data Availability Statement}
Data sharing is not applicable to this article because no datasets were generated or analyzed during the current study.

%\textbf{Compliance with ethical standards}

\textbf{Conflict of interest} The authors declare that they have no conflict of interest.

\textbf{Code availability}
Not applicable.

\textbf{AI Assistance Statement}
 AI is used as an assistive tool.

\end{document}